\documentclass[11pt]{amsart}

\usepackage[pagewise]{lineno}

\usepackage{latexsym, amsmath, amssymb,amsthm,amsopn,amsfonts}
\usepackage{version}
\usepackage{epsfig,graphics,color,graphicx,graphpap}
\usepackage{amssymb}
\usepackage{todonotes}
\usepackage{listings}
\usepackage{mathrsfs}
\usepackage[citecolor=blue]{hyperref}

\newtheorem{theorem}{Theorem}

\newtheorem{lemma}[theorem]{Lemma}
\newtheorem{corollary}[theorem]{Corollary}

\theoremstyle{definition}
\newtheorem{definition}[theorem]{Definition}
\newtheorem{remark}[theorem]{Remark}

\newcommand{\p}{\partial}

\title[]{Conformal boundary rigidity for simple Finsler metrics}

\author[K. Lam]{Kelvin Lam}
\address{Department of Mathematics, University of California Santa Barbara, Santa Barbara, CA 93106, USA}
\email{klam@math.ucsb.edu}

\author[H. Zhou]{Hanming Zhou}
\address{Department of Mathematics, University of California Santa Barbara, Santa Barbara, CA 93106, USA}
\email{hzhou@math.ucsb.edu}

\begin{document}

\begin{abstract}
{In this paper, we consider the conformal boundary rigidity problem for Finsler metrics on a bounded domain. We establish stability estimates for determining the conformal factors of simple Finsler metrics in the same conformal class from their boundary distance functions. Our results provide the Finslerian version of the well-known stability estimates for Riemannian metrics \cite{Muhometov}.}
\end{abstract}

\maketitle

\section{Introduction}

We consider the {\it boundary rigidity problem}, which arises naturally in geophysics in an attempt to determine the inner structure of the Earth, i.e., the sound speed or the index of refraction, from measurements of travel times of seismic waves on the Earth’s surface. As an approximate mathematical model, let $D\subset \mathbb R^n$, $n\geq 2$, be a simply connected closed bounded domain with smooth boundary $\p D$. Most studies of seismology assume that the Earth is an isotropic elastic medium, i.e., the sound speed only depends on positions, which induces a Riemannian metric $g=n^2(x)dx^2$ on $D$ with $n(x)$ the index of refraction.
This is generally true for the Earth’s crust and mantle. However, in the Earth’s deep interior, the presence of anisotropy has been widely recognized \cite{Silver96, RW2017}. In anisotropic linear elasticity, the wave propagation is governed by the elastic wave equation. This elastic system can be microlocally decoupled into 3 diﬀerent polarizations \cite{SDH2002}, which correspond to 3 positive eigenvalues of the principal symbol of the elastic wave operator. It turns out that the largest eigenvalue, corresponding to the fastest polarization, is homogeneous of degree 2 w.r.t. the frequency variable, so gives rise to a Finsler metric. Therefore, the travel time of the fastest polarized wave (called the qP wave) is characterized by the length of geodesics w.r.t. this Finsler metric. Due to the connection between the linear elasticity and the Finsler geometry, the study on inverse problems in Finsler geometry has gained increasing attention in recent years.

In the current paper, we study the boundary rigidity problem for Finsler metrics. Recall that a Finsler metric $F$ is a  nonnegative function on the tangent bundle $TD$, that is a Minkowski norm on the tangent plane $T_xD$ for each $x\in D$ satisfying
\begin{enumerate}
    \item $F(x,\lambda v)=\lambda F(x,v)$ for all $\lambda>0$ and $(x,v)\in TD$;
    \item For any $(x,v)\in TD\setminus \{0\}$, the symmetric bilinear form 
    $$g_v: T_xD\times T_xD\to \mathbb R,\quad g_v(u,w):=\frac{1}{2}\frac{\p^2}{\p s\p t}F^2(x,v+su+tw)|_{s=t=0}$$
    is positive definite.
\end{enumerate}
Given a Finsler metric $F$ on $D$, we define the associated distance function $\tau: D\times D\to [0,\infty)$ by 
$$\tau (x,y):=\inf \{\int_0^T F(\gamma(t),\dot\gamma(t))\,dt\, :\, \gamma:[0,T]\to D \,\mbox{is a Lipschitz continuous curve}, \gamma(0)=x, \gamma(T)=y\}$$
for any $x,y\in D$. The restriction of $\tau$ on the boundary, i.e. $\tau|_{\p D\times \p D}$, is called the boundary distance function. We ask the question: {\it Can we determine the Finsler metric $F$ from its boundary distance function?}

It is well-known that even in the Riemannian case, the boundary rigidity problem is not solvable in general, and one needs to impose additional geometric conditions. One such condition is simplicity. We say that a Finsler metric $F$ on $D$ is {\it simple} if 
\begin{enumerate}
    \item Every pair of points on $D$ is connected by a unique Finsler geodesic;
    \item The boundary $\p D$ is strictly convex w.r.t. $F$.
\end{enumerate}
See section \ref{preliminaries} for more details.
In the meantime, the boundary rigidity problem for general Riemannian metrics has a natural gauge: any isometry that preserves the boundary will also preserve the boundary distance function. However, Finsler metrics can admit non-isometric perturbations that preserve boundary distances \cite{ivanov2013}, which makes it even harder to characterize the non-uniqueness in the Finslerian case. As an initial step towards addressing the Finslerian boundary rigidity problem, we focus on the case of simple Finsler metrics of the same conformal class, namely
$$F(x,v)=\lambda (x) F_0(x, v)$$
for a fixed Finsler metric $F_0$ and positive scaler functions $\lambda$.

\begin{definition}
  Let $F_0$ be a $C^5$ (away from the zero section) \textit{reversible} Finsler metric on $D$ (so $F_0(x,v)=F_0(x,-v)$ for any $(x,v)\in TD$), and let $\lambda_0,\lambda_m$ and $M$ be positive numbers such that $\lambda_0<\lambda_m$. We define $\Lambda(\lambda_0,\lambda_m, M)$ to be the set of all functions $\lambda\in C^3(D)$ satisfying:
  \begin{enumerate}
      \item the metric $F=\lambda F_0$ is a simple Finsler metric on $D$;
      \item $\lambda_0\leq \lambda(x)\leq \lambda_m$ for all $x\in D$;
      \item Let $\exp_x(v)$ denote the exponential map w.r.t. $F=\lambda F_0$ based at $x\in D$, which is a $C^1$ diffeomorphism onto $D$, $C^2$ away from the zero section, due to the simplicity assumption (see e.g. \cite{Shen}). We assume that 
      $$\|D_y \exp_x^{-1}(y)\|<M$$ 
      for any $x,y\in D$, where $\|\cdot\|$ denotes the operator norm on $\mathbb R^n$.
  \end{enumerate}
\end{definition}


Let $\nabla_x \tau(x,y)$ denote the Euclidean differential of $\tau$ w.r.t. the first variable $x$, since our Finsler metric is reversible, $\nabla_x \tau(x,y)=\nabla_x \tau(y,x)$, knowing $\nabla_x \tau(x,y)$ is equivalent to knowing $\nabla \tau =(\nabla_x \tau,\nabla_y \tau)$. For the inverse problem, we are given the boundary restriction of $\tau$, i.e. $\tau|_{\p D\times\p D}$, so in practice the information that one can immediately obtain is the tangential derivatives of $\tau$, denoted by $\nabla^T \tau(x,y)$, for $x,y\in \p D$, which is equivalent to $\nabla (\tau|_{\p D\times\p D})$, the boundary differential.

In order to state our main result of the boundary rigidity for simple conformal Finsler metrics, we define the following weighted $L^1$-norm for $\nabla^T \tau$ on the boundary
\begin{equation}\label{weighted norm}
    \|\nabla^T\tau\|_{\mathcal L^1(\p D\times \p D)}:= \int_{\p D} \int_{\p D} |\nabla_x^T \tau(x,y)|\,\frac{dS_x dS_y }{|x-y|^{n-2}}
\end{equation}
where $dS_x$, $dS_y$ are the Euclidean surface measures of $\p D$ w.r.t. the variables $x$ and $y$ respectively. Throughout the paper, we denote $|\cdot|$ the standard Euclidean norm. The right hand side, as a singular integral, is integrable, see Lemma \ref{estimate1} in section \ref{preliminaries}.

\begin{theorem}\label{main theorem}
   Let $F_0$ be a $C^5$ reversible Finsler metric on $D$. Given any two functions $\lambda, \tilde\lambda\in \Lambda(\lambda_0,\lambda_m,M)$ with corresponding distance functions $\tau$ and $\tilde\tau$ respectively,
   \begin{equation}\label{stability estimate}
       \|\lambda-\tilde\lambda\|^2_{L^2(D)}\leq C \|\nabla^T(\tau-\tilde\tau)\|_{\mathcal L^1(\p D\times \p D)},
   \end{equation}
   where the constant $C$ depends only on $\lambda_0,\lambda_m,M,F_0$ and the domain $D$.
   
   In particular, given any two positive functions $\lambda, \tilde\lambda\in C^3(D)$ such that the Finsler metrics $\lambda F_0$ and $\tilde\lambda F_0$ are both simple, if $\tau|_{\p D\times\p D}=\tilde\tau|_{\p D\times\p D}$, then $\lambda=\tilde\lambda$.
\end{theorem}

To the best of our knowledge, the above theorem provides the first uniqueness and stability estimates for the conformal Finslerian boundary rigidity problem. Due to the connection between the
linear elasticity and the Finsler geometry, the study of inverse problems in Finsler geometry has been very active in recent years, see e.g. \cite{ivanov2013, ivanov2014,Yernat2018,dehoop2022,dehoop2023,dehoop2025, dehoop2026}
and the references therein.

When $F_0$ is a Riemannian metric, the boundary rigidity was proved for conformal simple metrics \cite{Muhometov2D,Mukhometov1975,MR1978}, and stability estimates have been established in \cite{Mukhometov1975, Muhometov,Beylkin79,sharafutdinov1994,TarikereZhou}. It is a conjecture that general simple Riemannian metrics are boundary distance rigid \cite{Michel}. This has been proved in dimension two \cite{PesUhl2005}. In dimensions $\geq 3$, this is known for generic simple metrics \cite{StefUhl2005}. When caustics appear, a completely new approach was established in \cite{SUV2016, SUV2021} for the boundary rigidity problem in dimensions $\geq 3$, assuming a convex
foliation condition. We refer to \cite{SUVZ2019} for summaries of recent developments on the
boundary rigidity problem.

To prove Theorem \ref{main theorem}, we follow Mukhometov's ideas for the Riemannian case \cite{Muhometov} closely, but also make crucial modifications to accommodate the presence of the Finsler metric $F_0$. Thanks to the homogeneity property of Finsler metrics, we are able to establish similar type of stability estimates. Note that the stability estimates in \cite{Muhometov} rely on the full differential $\nabla (\tau-\tilde\tau)$ and second order derivatives of $\tau-\tilde\tau$ on the boundary. However, we show that it's sufficient to use only the tangential derivatives $\nabla^T(\tau-\tilde\tau)$ as in \eqref{stability estimate}. See also \cite[Theorem 8.5.1]{sharafutdinov1994} for a similar estimate in the Riemannian case.

When the dimension $n=2$, the singular weight in \eqref{weighted norm} vanishes, so the right hand side of the stability estimate \eqref{stability estimate} reduces to the standard $L^1$ norm on $\p D\times \p D$. Since $\p D\times\p D$ has finite volume, we can further bound the $L^1$ norm by the $L^2$ norm, namely
\begin{equation}\label{2D stability}
    \|\lambda-\tilde\lambda\|^2_{L^2(D)}\leq C\|\nabla^T(\tau-\tilde\tau)\|_{L^2(\p D\times\p D)},
\end{equation}
which is consistent with the original work by Mukhometov \cite{Muhometov2D} on the conformal boundary rigidity for simple Riemannian surfaces.

In dimension $n\geq 3$, if $\lambda$ and $\tilde\lambda$ are equal near the boundary, one can improve estimate \eqref{stability estimate} to \eqref{2D stability} as well. Let $\mathcal O$ be an open subdomain of $D$ so that $\mathcal O\subset \overline{\mathcal O}\subset D^o$, where $D^o$ is the interior of $D$, and let $\rho\in \Lambda (\lambda_0,\lambda_m,M)$ be fixed. We define the following subset of $\Lambda(\lambda_0,\lambda_m,M)$
$$\Lambda_{\mathcal O,\rho}(\lambda_0,\lambda_m,M):=\{\lambda\in \Lambda(\lambda_0,\lambda_m,M)\,:\, \lambda|_{D\setminus \mathcal O}=\rho\}.$$

\begin{corollary}\label{fixed near boundary}
   Let $F_0$ be a $C^5$ reversible Finsler metric on $D\subset \mathbb R^n$, $n\geq 3$. Given any two functions $\lambda, \tilde\lambda\in \Lambda_{\mathcal O,\rho}(\lambda_0,\lambda_m,M)$ with corresponding distance functions $\tau$ and $\tilde\tau$ respectively,
   $$\|\lambda-\tilde\lambda\|^2_{L^2(D)}\leq C \|\nabla^T(\tau-\tilde\tau)\|_{L^2(\p D\times \p D)},$$
   where the constant $C$ depends only on $\lambda_0,\lambda_m,M,F_0, \rho, \mathcal O$ and $D$.
\end{corollary}

A similar setting for Riemannian metrics has been considered in \cite{TarikereZhou}.
\medskip

The structure of the paper is as follow: We first provide some geometric preliminaries for the Finsler geometry and inverse problems in section \ref{preliminaries} where we also characterize the singularity of the boundary distance functions. In section \ref{proof of main theorems} we prove the main results assuming a crucial integral identity, see Lemma \ref{lemma2}. 
We then prove the integral identity in section \ref{proof of lemma}.

\bigskip

\noindent {\bf Acknowledgments:} The authors would like to thank Prof John M Lee, Plamen Stefanov and Vladimir Sharafutdinov for helpful discussions. The authors are partly supported by the NSF grant DMS-2408369 and Simons Foundation Travel Support for Mathematicians MPS-TSM-00008046.

\section{Preliminaries}\label{preliminaries}

Given a Finsler metric $F$ on $D$, 
we denote
$$a_{ij}(x,v):=\frac{1}{2}\frac{\p^2}{\p v^i \p v^j} F^2(x,v),$$
then $a(x,v)=\sum_{i,j=1}^n a_{ij}(x,v)dx_i\otimes dx_j$ is called the {\it fundamental tensor} of $F$. Since $F$ induces a Minkowski norm on $T_xD$ for any $x\in D$, we get that $a(x,v)$ is positive definite, symmetric and homogeneous of degree 0 w.r.t $v$ on $TD\setminus \{0\}$. If $a$ is independent of $v$, it defines a Riemannian metric. We can define Finsler geodesics using the fundamental tensor $a$ in a way similar to the Riemannian case. In particular, Finsler geodesics are of constant speed, see e.g. \cite{Shen}. In this paper, we consider Finsler geodesics with unit speed.

We define the \textit{Legendre transformation}, which is the Finsler analog of the musical isomorphism for Riemannian metrics,
$$L:  TD\setminus \{0\}\to T^*D\setminus \{0\},\quad L(x,v):=a(x,v)\langle v,\cdot\rangle=(x, a_{ij}(x,v)v^i dx^j),$$
where $T^*D$ is the cotangent bundle.
Note that $L$ is a diffeomorphism. We define $F^*:=F\circ L^{-1}$ the co-Finsler metric on $T^*D$, whose fundamental tensor is denoted by $a^*$. Then $a$ and $a^*$ are connected by the following equality
$$a^{-1}=a^*\circ L,$$
where $a^{-1}$ is the inverse matrix of the fundamental tensor $a=(a_{ij})$. 
See \cite{Shen, Mdahl} for more details. We use $a^{ij}$ to denote the components of the fundamental tensor $a^*$ of the co-Finsler metric.

Now let $G \in C^2(D)$ be a boundary defining function of $\partial D$, where $G(x) < 0$ for $x \in D^o$, $G(x)>0$ for $x\in \mathbb R^n\setminus D$, $G(x) = 0$ and $|\nabla G (x)| = 1$ for $x \in \p D$. We say that $\p D$ is strictly convex w.r.t. a Finsler metric $F$ if any Finsler geodesic $\gamma(t)$ w.r.t. $F$ tangent to $\p D$ at $t=0$ satisfies that $\p_t^2 G(\gamma(t))|_{t=0}>0$.

Recall that a Finsler metric $F$ on $D$ is simple if every pair of points is connected by a unique Finsler geodesic and $\p D$ is strictly convex w.r.t. $F$. So given any $x, y\in D$, the distance $\tau(x,y)$ is realized by a unique distance minimizing geodesic $\gamma_{x,y}$ connecting $x, y$ as follows
\begin{equation}\label{distance}
\tau(x,y)=\int_{\gamma_{x,y}} F(\gamma_{x,y}(t),\dot\gamma_{x,y}(t))\,dt.
\end{equation}
 Due to the simplicity assumption, the exponential map $\exp_x$ is a diffeomorphism onto $D$ at any $x\in D$, then $\tau(x,y)$ is also equal to $F(\exp_x^{-1}(y))$.

Now we consider the conformal Finsler metric $F=\lambda F_0$, with $a$ the fundamental tensor of the fixed reversible Finsler metric $F_0$. By choosing an appropriate smooth geodesic variation and applying the first variation formula for Finsler metrics \cite[equation 5.6]{Shen}, we have the following equality:
\begin{equation}\label{Icanoeq}
    \tau_{x_i} (x,y)  = \lambda^2(x) a_{ij}(x,v(y,x)) v^j( y,x) ,  \ \ \ \  i=1,2,....n;
\end{equation}
where $v(y,x) = (v^1(y,x),..., v^n(y,x))$ is the unit vector based at $x$ tangent to the geodesic connecting $x$ and $y$, oriented away from $y$ and satisfies $\lambda^2(x) a_{ij}(x, v(y,x)) v^iv^j = 1$. 
Clearly we have $v(y,x) = -\frac{\exp^{-1}_x(y)}{\tau(x,y)}$.
Note that \eqref{Icanoeq} is the Legendre transformation of $v$ in coordinates, in other words, $L(x,v(y,x))=\nabla_x \tau(x,y)$.

The inverse of $L$ gives the following relations:
\begin{equation}\label{eq3}
    v^i (y,x) = \lambda^{-2}(x) a^{ij}(x, \nabla_x\tau(x,y)) \tau_{x_j},
\end{equation}
\begin{equation}\label{eq2}
    \tau_{x_i}(x,y) v^i(y,x) = 1,
\end{equation}
\begin{equation}\label{ellipsoid}
    \lambda^{-2}(x)a^{ij}(x, \nabla_x\tau(x,y))\tau_{x_i} \tau_{x_j} = 1.
\end{equation}

We prove the following lemma that determines the singularity type of $\tau$ at the diagonal $x=y$:

\begin{lemma}\label{estimate1}
    Let $\tau$ be the distance function w.r.t. the Finsler metric $F=\lambda F_0$ for $\lambda\in \Lambda(\lambda_0,\lambda_m,M)$, then $\tau(x,y)$ is continuous on $D \times D$, twice differentiable off the diagonal. Moreover, there exist $C_1 , C_2 > 0 $ depending on $\lambda_0, \lambda_m, M, F_0$ and the diameter of $D$ such that:
    \begin{equation}\label{est1}
        |\tau_{x_i}| \leq C_1, \ \ |\tau_{y_i}| \leq C_1
    \end{equation}
    \begin{equation}\label{est2}
        |\tau_{x_i y_j}| \leq C_2|x-y|^{-1}
    \end{equation}
where $|\cdot|$ denotes the Euclidean norm.
\end{lemma}

\begin{proof}
    Since Finslerian geodesics are of constant speed, $\tau(x,y) = \lambda (x) F_0(\exp^{-1}_x(y))$. Notice that $\lambda\in C^3(D)$ and $F_0$ is continuous on $TD$ and $C^5$ away from the zero section, so $\exp_x^{-1}$ is continuous on $D$ and $C^1$ on $D\setminus \{x\}$ \cite[theorem 11.1.1]{Shen}, thus $\tau$ is also continuous on $D\times D$.

Since $\lambda_0\leq \lambda \leq \lambda_m$ and $a_{ij}$ is bounded over the sphere bundle $\{(x,v) \in TD\, |\, F_0(x,v)=1 \}$, it is bounded over the sphere bundle 
$$SD=SD_{\lambda} := \{(x,v) \in TD\,|\, \lambda F_0(x,v)=1 \}$$
since $a_{ij}$ is 0-homogeneous. On the other hand, it's easy to see that $|v|$ is uniformly bounded for any $(x,v)\in SD$.
Estimate \eqref{est1} then follows from \eqref{Icanoeq}.

    Now \eqref{distance} implies that $\tau(x,y) \geq C |x-y|$ for $C>0$ that depends on the diameter of $D$, $\lambda_0$, and $F_0$. Hence we have 
    \begin{equation}\label{normest}
    \frac{1}{|\tau(x,y)|}\leq \frac{1}{C|x-y|} .
    \end{equation}
    By reversibility of $F$: $$v(y,x) = -\frac{\exp^{-1}_x(y)}{\tau(x,y)},$$ 
    by the mean value theorem $|\exp_x^{-1}(y)|$ is bounded by a constant that depends on the diameter of $D$ and $M$ since $\exp_x^{-1}(x)=0$. 

    Substituting the above equation into \eqref{Icanoeq}, we obtain
    \begin{equation} \label{tauzetabd}
        \tau_{x_i} (x, y)  = -\lambda^2(x) a_{ij}(x,v(y,x))) \frac{(\exp^{-1}_x(y))^j}{\tau(x,y)}.
    \end{equation}
Differentiating with respect to $y_k$ gives:
    \begin{equation}
\begin{split}
-\tau_{x_i y_k} 
&=
\lambda^2(x) \partial_{v^l} a_{ij}(x,v(y,x)))  \frac{ \partial_{y_k}(\exp^{-1}_x(y))^l}{\tau(x,y)} v^j(y,x) - \lambda^2(x) \partial_{v^l} a_{ij}(x,v(y,x)))  \frac{\tau_{y_k}}{\tau(x,y)} v^l(y,x) v^j(y,x)
\\
&+
\lambda^2(x) a_{ij}(x,v(y,x)))  \frac{ \partial_{y_k}(\exp^{-1}_x(y))^j}{\tau(x,y)} - \lambda^2(x)  a_{ij}(x,v(y,x)))  \frac{\tau_{y_k}}{\tau(x,y)} v^j(y,x).
\end{split}
\end{equation}
By the Euler's theorem, the first two terms involving $\p_v a_{ij}$ vanish. Since $a_{ij}(x,v(y,x)))$ is uniformly bounded
, $\lambda\in \Lambda(\lambda_0,\lambda_m,M)$ so $\partial_{y}\exp^{-1}_x(y)$ is bounded. Together with \eqref{est1} and \eqref{normest}, this completes the proof.
\end{proof}

Recall the $\mathcal L^1$-norm defined in the introduction
\begin{equation}
    \|\nabla^T\tau\|_{\mathcal L^1(\p D \times \p D)} = \int_{\p D}  \int_{\p D} |\nabla^T_x \tau(x,y)|\, \frac{dS_x dS_y}{|x-y|^{n-2}},
\end{equation}
by Lemma \ref{estimate1} we see that the integrand has a singularity of the type $|x-y|^{-(n-2)}$ which is integrable, thus the norm is well defined.

\section{Proof of main Theorems}\label{proof of main theorems}

The main ingredient in the proof of Theorem \ref{main theorem} is the following lemma. The Riemannian version was established by Mukhometov in \cite{Muhometov}. Given a matrix $M$, we denote its determinant by $\det M:=|M|$.

\begin{lemma}\label{lemma2}
Let $F_0$ be a $C^5$ reversible Finsler metric on $D$, and let $\lambda\in \Lambda(\lambda_0,\lambda_m,M)$ with corresponding distance functions $\tau$.
    Let $w(x,y)$ be a continuous function on $D \times D$ with continuous partial derivatives in $x$ and $y$ up to order three away from the diagonal, with first derivatives bounded and second derivatives having possibly a $|x-y|^{-1}$ singularity. Then the following identity holds:
\begin{equation}\label{mainlemma2}
\begin{split}
2(-1)^n \int_{\p D} dS_{y} \int_D B(\mathcal{L}w, w, \tau, \tau) dx = \int_D dx\int_{S_x^*D}\eta( \frac{v}{|v|})\bigg[(n-2) (\mathcal{L}w)^2 \\
+
\lambda^{-2}(x)a^{ij}(x, \eta) w_{x_i}w_{x_j}\bigg]_{y= y(x,\eta)} dS_\eta  + (-1)^n \int_{\p D} dS_y \int_{\p D} B(w,w,G,\tau) dS_x.
\end{split}
\end{equation}
where $\mathcal{L}w = w_{x_i} v^i$, and 
\begin{equation}\label{multilinear form}
    B(f, \phi, \psi, \tau)(x,y) =   \begin{vmatrix}
    0 & 0 & \phi_{x_1} & \cdots & \phi_{x_n} \\
    0 & 0 & \psi_{x_1} & \cdots & \psi_{x_n} \\
    f_{y_1} & G_{y_1}& \tau_{x_1 y_1} & \cdots & \tau_{x_n y_1} \\
    \vdots & \vdots & \vdots & \ddots & \vdots \\
   f_{y_n} & G_{y_n}& \tau_{x_1 y_n} & \cdots & \tau_{x_n y_n} \\ 
\end{vmatrix}. 
\end{equation}
Here $G$ is the boundary defining function, $v=v(y,x)$ is the unit vector w.r.t. the metric $\lambda F_0$. $S^*_x D\subset T^*_x D \cong \mathbb{R}^n$ is the set of unit covectors at $x$, 
and $dS_\eta$ is the surface measure of $S^*_x D$ as a subset of $\mathbb{R}^n$. $y(x,\eta)$ is the inverse of the diffeomorphism  $\eta = \nabla_x \tau (x,y)$ , $y \in \p D$ for a fixed $x \in D$.
\end{lemma}

By Lemma \ref{estimate1} and the regularity assumption on $w$, $B(\mathcal L w,w,\tau,\tau)$ has a singularity of the type $|x-y|^{n-1}$, and $B(w,w,G,\tau)$ has a singularity of the type $|x-y|^{n-2}$, therefore all the integrals in \eqref{mainlemma2} are well-defined.

We defer the proof of Lemma \ref{lemma2} to section \ref{proof of lemma}. Applying Lemma \ref{lemma2}, we can prove the following results which are crucial in the proof of the main theorems.

\begin{lemma}\label{inequality 1}
    Let $F_0$ be a $C^5$ reversible Finsler metric on $D$. Given any two functions $\lambda, \tilde\lambda\in \Lambda(\lambda_0,\lambda_m,M)$ with corresponding distance functions $\tau$ and $\tilde\tau$ respectively, we have the following inequality:
    \begin{equation}
        C\int_D \lambda^{n-1}(\lambda-\tilde{\lambda})\,dx \leq \frac{(-1)^{n-1}}{n-1} \int_{\p D} dS_x\int_{\p D} B(\tau-\tilde{\tau}, \tau, G, \tau)\, dS_y
    \end{equation}
    where $C$ is a positive constant depending on $F_0$.
\end{lemma}

\begin{proof}
Notice that
$$\mathcal L \tau=\tau_{x_i}v^i=1,$$
where $v=v(y,x)$ is the unit vector w.r.t. the metric $\lambda F_0$.
By the definition of the multilinear form $B$, we have 
    $$ B(\mathcal L(\tau-\tilde\tau),\tau-\tilde\tau,\tau,\tau) + B(\mathcal L \tau,\tau,\tau,\tau)-B(\mathcal L \tilde\tau,\tilde\tau,\tau,\tau) =0$$
By Lemma \ref{lemma2}, let $w$ be $\tau-\tilde\tau$, $\tau$ and $\tilde\tau$ respectively, the above equation implies that
\begin{equation}\label{eq 16}
\begin{split}
    \int_{D} dx \int_{S_x^*D} \eta (\frac{v}{|v|}) \bigg[(n-2) (\mathcal L(\tau-\tilde\tau))^2 + \lambda^{-2}a^{ij}(x,\eta)(\tau-\tilde\tau)_{x_i}(\tau-\tilde\tau)_{x_j} +(n-1) \\
    -(n-2)(\mathcal L\tilde{\tau})^2-\lambda^{-2}a^{ij}(x,\eta)\tilde\tau_{x_i}\tilde\tau_{x_j}\bigg]_{y = y(x,\eta)} dS_\eta \\
    =(-1)^{n-1}\int_{\p D} \int_{\p D} B(\tau-\tilde\tau, \tau-\tilde\tau, G, \tau) + B(\tau,\tau,G,\tau) -B(\tilde{\tau},\tilde{\tau},G,\tau)\,dS_x dS_y
\end{split}
\end{equation}

It's easy to check that the integrand on the RHS of \eqref{eq 16} satisfies
\begin{equation}\label{82}
    B(\tau-\tilde\tau, \tau-\tilde\tau, G, \tau) + B(\tau,\tau, G,\tau) - B(\tilde{\tau}, \tilde{\tau}, G, \tau) = B(\tau-\tilde{\tau}, \tau, G, \tau ) + B(\tau, \tau-\tilde{\tau}, G, \tau )
\end{equation}
The structure of $B$ also suggests that
$$\int_{\p D} \int_{\p D} B(\tau-\tilde{\tau}, \tau, G, \tau ) \,dS_x dS_y=\int_{\p D} \int_{\p D} B(\tau, \tau-\tilde{\tau}, G, \tau ) \,dS_x dS_y$$
Therefore, the RHS of \eqref{eq 16} is equal to
$$2(-1)^{n-1}\int_{\p D}\int_{\p D} B(\tau-\tilde{\tau}, \tau, G, \tau ) \,dS_x dS_y.$$

To analyze the LHS of \eqref{eq 16}, notice that
\begin{align*}
    (\mathcal L(\tau-\tilde\tau))^2 & = 1-2\mathcal L\tau \mathcal L\tilde{\tau}+(\mathcal L\tilde{\tau})^2\\
\lambda^{-2}a^{ij}(x,\eta)(\tau-\tilde\tau)_{x_i}(\tau-\tilde\tau)_{x_j} & = 1 + \lambda^{-2}a^{ij}(x,\eta) \tilde\tau_{x_i}\tilde\tau_{x_j}- 2\lambda^{-2}a^{ij}(x,\eta)\tau_{x_i}\tilde\tau_{x_j}\\
\mathcal L\tau \mathcal L\tilde{\tau} & = \mathcal L\tilde{\tau} = \tilde{\tau}_{x_i}v^i=\lambda^{-2}a^{ij}(x,\eta)\tau_{x_i}\tilde\tau_{x_j}.
\end{align*}
Thus the LHS of \eqref{eq 16} is equal to 
$$2(n-1)\int_D dx \int_{S_x^*D} \eta (\frac{v}{|v|}) (1-\lambda^{-2} a^{ij}(x,\eta) \tau_{x_i}\tilde{\tau}_{x_j} )\bigg|_{y = y(x,\eta)} dS_\eta.$$
Since $\nabla_x \tau= \eta$, we apply Cauchy's inequality (\cite[Lemma 1.2.3]{Shen}) to conclude:
\begin{equation} \label{86}
    a^{ij}(x,\nabla_x \tau) \tau_{x_i}\tilde{\tau}_{x_j}  \leq F_0^*(x,\nabla_x \tau)F_0^*(x,\nabla_x \tilde{\tau}) = \lambda \tilde{\lambda},
\end{equation}
where the last equality follows since $\nabla_x \tau\in S_x^*D_\lambda$ and $\nabla_x \tilde{\tau}\in S_x^*D_{\tilde\lambda}$. 
Applying \eqref{eq3}, it's not difficult to check that the vector $v$ is outward normal to $S^*D_\lambda$ (see also the proof of Lemma \ref{auxilarylemma1}),
so we have 
\begin{equation}\label{87}
\eta (\frac{v}{|v|}) > 0 
\end{equation}
By \eqref{86} and \eqref{87}, we obtain 
$$\int_D dx \int_{S_x^*D} \eta (\frac{v}{|v|}) (1-\lambda^{-2} a^{ij}(x,\eta) \tau_{x_i}\tilde{\tau}_{x_j} )\bigg|_{y = y(x,\eta)} dS_\eta\geq  \int_D (1-\frac{\tilde{\lambda}}{\lambda})dx \int_{S_x^*D} \eta (\frac{v}{|v|}) dS_\eta$$
Applying the divergence theorem to the last integral on the RHS, 
$$\int_{S_x^*D} \eta (\frac{v}{|v|})\,dS_\eta= n\int_{B^*_xD} d \eta_1\cdots d\eta_n,$$
where $B^*_xD=\{\eta\in T^*_xD\,:\, \lambda^{-1} F_0^* (x,\eta )\leq 1\}$ is viewed as a ball in $\mathbb R^n$. Note that $F_0^*(x,\eta)$ is homogeneous of degree 1 in $\eta$, it's clear that
$$\int_{B^*_xD} d\eta_1\cdots d\eta_n\geq \lambda^n(x)C(F_0)$$
where $C(F_0)$ is a constant depending on $F_0$, independent of $\lambda$.

It follows that
$$(n-1)n\, C(F_0)\int_D \lambda^{n-1}(\lambda-\tilde\lambda)\,dx\leq (-1)^{n-1}\int_{\p D}\int_{\p D} B(\tau-\tilde\tau,\tau, G,\tau)\,dS_x dS_y,$$
which completes the proof.
\end{proof}

\begin{lemma}
Under the same assumptions of Lemma \ref{inequality 1},
    \begin{equation}\label{inequality 2}
       C \int_D (\lambda^{n-1}-\tilde{\lambda}^{n-1})(\lambda-\tilde{\lambda})\, dx \leq \frac{(-1)^{n-1}}{(n-1)n} \int_{\p D} dS_x\int_{\p D} B(\tau-\tilde{\tau}, \tau, G, \tau) - B(\tau-\tilde{\tau}, \tilde{\tau}, G, \tilde{\tau}) dS_y
    \end{equation}
    where $C>0$ is a constant depending on $F_0$.
\end{lemma}

\begin{proof}
    The inequality is obtained from Lemma \ref{inequality 1} by interchanging the roles of $\lambda$ and $\tilde{\lambda}$.
\end{proof}

\begin{proof}[Proof of Theorem \ref{main theorem}]
    It's easy to see that the LHS of the inequality \eqref{inequality 2} is bounded below by 
    $$(n-1)C\lambda_0^{n-2}  \int_D(\lambda-\tilde{\lambda})^2\, dx=(n-1)C\lambda_0^{n-2}\|\lambda-\tilde\lambda\|^2_{L^2(D)}.$$

    For the RHS of \eqref{inequality 2}, note the following equality:
    $$B(\tau-\tilde{\tau}, \tau, G, \tau) -B(\tau-\tilde{\tau}, \tilde{\tau}, G, \tilde{\tau}) = \sum_{i=1}^n M_i $$
where 
$$M_i =   \begin{vmatrix}
      0 & 0 & \tilde{\tau}_{x_1} & ... & \tilde{\tau}_{x_{i-1}} & (\tau-\tilde\tau)_{x_i} & \tau_{x_{i+1}} & ... & \tau_{x_n} \\
      0 & 0 & G_{x_1} & ... & G_{x_{i-1}} & 0 & G_{x_{i+1}} & ... & G_{x_n} \\
      (\tau-\tilde\tau)_{y_1} & G_{y_1} & \tilde{\tau}_{x_1 y_1} & ... & \tilde{\tau}_{x_{i-1} y_1} & (\tau-\tilde\tau)_{x_i y_1} & \tau_{x_{i+1} y_1} & ... & \tau_{x_n y_1} \\
      \vdots & \vdots & \vdots & \vdots & \vdots & \vdots & \vdots & \vdots & \vdots\\
      (\tau-\tilde\tau)_{y_n} & G_{y_n} & \tilde{\tau}_{x_1 y_n} & ... & \tilde{\tau}_{x_{i-1} y_n} & (\tau-\tilde\tau)_{x_i y_n} & \tau_{x_{i+1} y_n} & ... & \tau_{x_n y_n} 
      \end{vmatrix} $$
which can be shown by induction. 

Then one can check that 
\begin{align}\label{M_i}
     M_i=\sum_{i\neq j, j\neq k} m_{j,k}\, \begin{vmatrix}
      (\tau-\tilde\tau)_{y_1} & G_{y_1} & \cdots & \alpha_{x_{p-1} y_1} & \alpha_{x_{p+1} y_1} & \cdots & \alpha_{x_{q-1} y_1} & \alpha_{x_{q+1} y_1} & \cdots \\
      \vdots & \vdots & \vdots & \vdots & \vdots & \vdots & \vdots & \vdots & \vdots \\
      (\tau-\tilde\tau)_{y_n} & G_{y_n} & \cdots & \alpha_{x_{p-1} y_n} & \alpha_{x_{p+1} y_n} & \cdots & \alpha_{x_{q-1} y_n} & \alpha_{x_{q+1} y_n} & \cdots 
      \end{vmatrix}
\end{align}
where $p=j$, $q=k$ if $j<k$, and $p=k$, $q=j$ if $j>k$. Here $\alpha_{x_\ell y_r}=\tilde\tau_{x_\ell y_r}$ if $\ell<i$, $\alpha_{x_\ell y_r}=(\tau-\tilde\tau)_{x_\ell y_r}$ if $\ell=i$, and $\alpha_{x_\ell y_r}=\tau_{x_\ell y_r}$ if $\ell>i$. $|m_{j,k}|=|G_{x_j}\cdot \alpha_{x_k}|$.

Since $G$ is the boundary defining function, the second column of the above matrix $\nabla_y G$ is normal to the boundary $\p D$, therefore the determinant in \eqref{M_i} only depends on the tangential part of the first column, i.e. we can replace the first column, which is $\nabla_y (\tau-\tilde\tau)$, by the tangential derivative $\nabla^T_y(\tau-\tilde\tau)$.

By Lemma \ref{estimate1}, the $\tau_{x_i}$, $\tilde{\tau}_{x_i}$ terms are bounded by a uniform constant, $\tau_{x_i y_j}$, $\tilde\tau_{x_i y_j}$ are bounded by $\frac{C}{|x-y|}$ where $C$ is a constant that depends on $\lambda_0, \lambda_m , M$, $F_0$ and the diameter of $D$. $G_{x_i}$ and $G_{y_j}$ are clearly bounded on $\p D$. Therefore, we can rewrite $M_i$ as
$$M_i=\sum_{i\neq j, j\neq k} m_{j,k}\, \nabla^T_y(\tau-\tilde\tau) (\beta_{j,k}),$$
where $\beta_{j,k}$ is a vector whose entries are bounded by $\frac{C}{|x-y|^{n-2}}$ for some constant $C$. It follows that
$$|M_i|\leq C \frac{|\nabla^T_y (\tau-\tilde\tau)|}{|x-y|^{n-2}},\quad i=1,\cdots, n$$
for some constant $C$ that depends on $\lambda_0, \lambda_m, M$, $F_0$ and the diameter of $D$.
Thus
\begin{align*}
    |B(\tau-\tilde{\tau}, \tau, G, \tau) -B(\tau-\tilde{\tau}, \tilde{\tau}, G, \tilde{\tau})| \leq nC \frac{|\nabla^T_y (\tau-\tilde\tau)|}{|x-y|^{n-2}}
\end{align*}
It follows that the RHS of \eqref{inequality 2} is bounded above by 
$$C \|\nabla^T(\tau-\tilde\tau)\|_{\mathcal L^1(\p D\times\p D)}$$
for some constant $C>0$.

This proves Theorem \ref{main theorem}.
\end{proof}

\begin{remark}
    As observed in the proof of Theorem \ref{main theorem}, the weighted $L^1$-norm is derived mainly to control the determinants $M_i$, $i=1,\cdots n$. In dimension $2$, Theorem \ref{main theorem} implies that it's enough to use the standard $L^1$-norm, namely
    $$\|\lambda-\tilde\lambda\|^2_{L^2(D)}\leq C\|\nabla^T(\tau-\tilde\tau)\|_{L^1(\p D\times\p D)}.$$
    We can also calculate $M_i$ explicitly in this case. Note that when dimension $n=2$, 
\begin{align*}
    M_1 & = \begin{vmatrix}
    0 & 0 & (\tau-\tilde\tau)_{x_1} & \tau_{x_2} \\
    0 & 0 & 0 & G_{x_2} \\
    (\tau-\tilde\tau)_{y_1} & G_{y_1} & (\tau-\tilde\tau)_{x_1 y_1} & \tau_{x_2 y_1} \\
    (\tau-\tilde\tau)_{y_2} & G_{y_2} & (\tau-\tilde\tau)_{x_1 y_2} & \tau_{x_2 y_2}
\end{vmatrix}\\
  & = (\tau-\tilde\tau)_{x_1}G_{x_2}\bigg( (\tau-\tilde\tau)_{y_1}G_{y_2}-(\tau-\tilde\tau)_{y_2}G_{y_1}\bigg).
\end{align*}
Similarly, we get
$$M_2=-(\tau-\tilde\tau)_{x_2}G_{x_1}\bigg( (\tau-\tilde\tau)_{y_1}G_{y_2}-(\tau-\tilde\tau)_{y_2}G_{y_1}\bigg).$$ 
Therefore
$$M_1+M_2=\bigg((\tau-\tilde\tau)_{x_1}G_{x_2}-(\tau-\tilde\tau)_{x_2}G_{x_1} \bigg)\bigg( (\tau-\tilde\tau)_{y_1}G_{y_2}-(\tau-\tilde\tau)_{y_2}G_{y_1}\bigg).$$
Note that $\nu=\nabla G=(G_{x_1}, G_{x_2})$, thus
$$(\tau-\tilde\tau)_{x_1}G_{x_2}-(\tau-\tilde\tau)_{x_2}G_{x_1}=\nabla_x(\tau-\tilde\tau)(\nu^\perp)=\nabla^T_x(\tau-\tilde\tau) (\nu^\perp).$$
It follows that
\begin{align*}
    \|\lambda-\tilde\lambda\|^2_{L^2(D)} & \leq C \int_{\p D}\int_{\p D} |\nabla^T_x(\tau-\tilde\tau)|\cdot |\nabla_y^T(\tau-\tilde\tau)|\, dS_x dS_y\\
    & \leq \frac{C}{2}\int_{\p D\times \p D} |\nabla^T_x(\tau-\tilde\tau)|^2+|\nabla_y^T(\tau-\tilde\tau)|^2\, dS_xdS_y\\
    & =C\|\nabla^T(\tau-\tilde\tau)\|^2_{L^2(\p D\times\p D)},
\end{align*}
so
$$\|\lambda-\tilde\lambda\|_{L^2(D)}\leq C \|\nabla^T(\tau-\tilde\tau)\|_{L^2(\p D\times \p D)}.$$
\end{remark}

\begin{proof}[Proof of Corollary \ref{fixed near boundary}]
    Given $\lambda, \tilde\lambda\in \Lambda_{\mathcal O,\rho}(\lambda_0,\lambda_m, M)$, so $\lambda=\tilde\lambda$ in $D\setminus \mathcal O$, which is a neighborhood of the boundary $\p D$. Then due to the convexity assumption of $\p D$, there exists a small constant $\delta>0$, so that $\tau(x,y)=\tilde\tau(x,y)$ for $x,y\in \p D$ if $|x-y|\leq \delta$ (so the unique geodesic connecting $x$ and $y$ stays entirely in $D\setminus \mathcal O$). 
    It follows that 
    \begin{align*}
        \|\lambda-\tilde\lambda\|^2_{L^2(D)} & \leq C \|\nabla^T (\tau-\tilde\tau)\|_{\mathcal L^1(\p D\times\p D)}\\
        & = C\int_{\{(x,y)\in \p D\times\p D\,:\, |x-y|>\delta\}} \frac{|\nabla^T_x(\tau-\tilde\tau)|}{|x-y|^{n-2}}\,dS_x dS_y\\
        & \leq \frac{C}{\delta^{n-2}}\|\nabla^T(\tau-\tilde\tau)\|_{L^1(\p D\times\p D)}\\
        & \leq C'\|\nabla^T(\tau-\tilde\tau)\|_{L^2(\p D\times\p D)}.
    \end{align*}
\end{proof}

\section{Proof of Lemma \ref{lemma2}}\label{proof of lemma}

This section is devoted to the proof of Lemma \ref{lemma2}.

We start with the LHS of \eqref{mainlemma2}.
To simplify the notation, we denote $B:=B(\mathcal Lw,w,\tau,\tau)$. By the product rule on derivatives of determinants, we obtain the following equation:
$$nB=-\sum_{i=1}^n(\sum_{i \neq j } A_{ij} - \frac{\partial}{\partial y_i}B_i +C_i +D_i )$$
where 
\begin{align*}
    A_{ij}= \begin{vmatrix}
     0 & 0 & w_{x_1} & \cdots  & w_{x_n} \\
     0 &  0 & \tau_{x_1} & \cdots & \tau_{x_n} \\
     (\mathcal L w)_{y_1} & G_{y_1}& \tau_{x_1 y_1} & \cdots & \tau_{x_n y_1} \\
    \vdots & \vdots & \vdots & \vdots & \vdots \\
    (\mathcal L w)_{y_{i-1}} &  G_{y_{i-1}}& \tau_{x_1 y_{i-1}} & \cdots & \tau_{x_n y_{i-1}} \\
    \mathcal L w  &  G &  \tau_{x_1 } & \cdots & \tau_{x_n } \\
    (\mathcal L w)_{y_{i+1}} &  G_{y_{i+1}}& \tau_{x_1 y_{i+1}} & \cdots & \tau_{x_n y_{i+1}} \\
    \vdots &  \vdots & \vdots & \vdots & \vdots \\    
    (\mathcal L w)_{y_{j-1}} &  G_{y_{j-1}}& \tau_{x_1 y_{j-1}} & \cdots & \tau_{x_n y_{j-1}} \\
    (\mathcal Lw)_{y_j y_i}  &  G_{y_j y_i} &  \tau_{x_1 y_j y_i} & \cdots & \tau_{x_n y_j y_i } \\
    (\mathcal Lw)_{y_{j+1}} &  G_{y_{j+1}}& \tau_{x_1 y_{j+1}} & \cdots & \tau_{x_n y_{j+1}} \\
     \vdots & \vdots & \vdots & \vdots & \vdots \\
    (\mathcal Lw)_{y_n} &  G_{y_{n}}& \tau_{x_1 y_{n}} & \cdots & \tau_{x_n y_{n}}
\end{vmatrix} ,\quad B_i= \begin{vmatrix}
     0 & 0 & w_{x_1} & \cdots & w_{x_n} \\
     0 &  0 & \tau_{x_1} & \cdots & \tau_{x_n} \\
     (\mathcal L w)_{y_1} & G_{y_1}& \tau_{x_1 y_1} & \cdots & \tau_{x_n y_1} \\
    \vdots & \vdots & \vdots & \vdots & \vdots \\
    (\mathcal Lw)_{y_{i-1}} &  G_{y_{i-1}}& \tau_{x_1 y_{i-1}} & \cdots & \tau_{x_n y_{i-1}} \\
    \mathcal Lw  &  G &  \tau_{x_1 } & \cdots & \tau_{x_n } \\
    (\mathcal Lw)_{y_{i+1}} &  G_{y_{i+1}}& \tau_{x_1 y_{i+1}} & \cdots & \tau_{x_n y_{i+1}} \\
    \vdots &  \vdots & \vdots & \vdots & \vdots \\
    (\mathcal Lw)_{y_n} &  G_{y_{n}}& \tau_{x_1 y_{n}} & \cdots & \tau_{x_n y_{n}} 
\end{vmatrix},
\end{align*}
    
\begin{align*}    C_i= \begin{vmatrix}
     0 & 0 & w_{x_1 y_i} & \cdots & w_{x_n y_i} \\
     0 &  0 & \tau_{x_1} & \cdots & \tau_{x_n} \\
     (\mathcal Lw)_{y_1} & G_{y_1}& \tau_{x_1 y_1} & \cdots & \tau_{x_n y_1} \\
    \vdots & \vdots & \vdots & \vdots & \vdots \\
    (\mathcal Lw)_{y_{i-1}} &  G_{y_{i-1}}& \tau_{x_1 y_{i-1}} & \cdots & \tau_{x_n y_{i-1}} \\
    \mathcal Lw  &  G &  \tau_{x_1 } & \cdots & \tau_{x_n } \\
    (\mathcal Lw)_{y_{i+1}} &  G_{y_{i+1}}& \tau_{x_1 y_{i+1}} & \cdots & \tau_{x_n y_{i+1}} \\
    \vdots &  \vdots & \vdots & \vdots & \vdots \\
    (\mathcal Lw)_{y_n} &  G_{y_{n}}& \tau_{x_1 y_{n}} & \cdots & \tau_{x_n y_{n}}
\end{vmatrix},\quad D_i= \begin{vmatrix}
     0 & 0 & w_{x_1} & \cdots & w_{x_n} \\
     0 &  0 & \tau_{x_1 y_i} & \cdots & \tau_{x_n y_i} \\
     (\mathcal Lw)_{y_1} & G_{y_1}& \tau_{x_1 y_1} & \cdots & \tau_{x_n y_1} \\
    \vdots & \vdots & \vdots & \vdots & \vdots \\
    (\mathcal Lw)_{y_{i-1}} &  G_{y_{i-1}}& \tau_{x_1 y_{i-1}} & \cdots & \tau_{x_n y_{i-1}} \\
    \mathcal Lw  &  G &  \tau_{x_1 } & \cdots & \tau_{x_n } \\
    (\mathcal Lw)_{y_{i+1}} &  G_{y_{i+1}}& \tau_{x_1 y_{i+1}} & \cdots & \tau_{x_n y_{i+1}} \\
    \vdots &  \vdots & \vdots &  \vdots & \vdots \\
    (\mathcal Lw)_{y_n} &  G_{y_{n}}& \tau_{x_1 y_{n}} & \cdots & \tau_{x_n y_{n}}
\end{vmatrix}.
\end{align*}
Here $B_i$ is obtained from $B$ by replacing the $(i+2)$-th row with $(\mathcal{L}w, G, \tau_{x_1}, \cdots, \tau_{x_n})$ (i.e. omitting the derivative with respect to $y_i$). 
$C_i$, $D_i$ and $A_{ij}$ are obtained from $B_i$ by relocating the derivative in $y_i$ of the elements from the $(i+2)$-th row to the first, second and $(j+2)$-th row respectively.

It follows that
\begin{equation}\label{main2}
    2(-1)^n \int_{\p D}  B \,dS_y = (-1)^{n-1} \int_{\p D}  [ (n-2) B + \sum_{i=1}^n(\sum_{i \neq j } A_{ij} - \frac{\partial}{\partial y_i}B_i +C_i +D_i ) ]\, dS_y
\end{equation}
Now following identically the arguments from \cite{Muhometov}, which are purely algebraic, we have that
\begin{align*}
     \sum_{i=1}^n \sum_{i \neq j} A_{ij} =0; \quad (A_{ij}=-A_{ji})\\
      \int_{\p D} (\sum_{i=1}^n \frac{\partial}{\partial y_i} B_i)\, dS_y = \int_{\p D} B\, dS_y.
\end{align*}
The second equality is essentially an application of the divergence theorem and the fact that $G=0$ on $\p D$.

Let $(\mathcal{L}w)_{y_k} = H_k + I_k$ where $H_k = w_{x_i y_k} v^i$ and $I_k = w_{x_i} v^i_{y_k}$, then we can decompose $B$ into $H+I$ 
where 
\begin{align*}
    H= \begin{vmatrix}
     0 & 0 & w_{x_1} & \cdots & w_{x_n} \\
     0 &  0 & \tau_{x_1} & \cdots & \tau_{x_n} \\
     H_1 & G_{y_1}& \tau_{x_1 y_1} & \cdots & \tau_{x_n y_1} \\
    \vdots & \vdots & \vdots & \vdots & \vdots \\
    H_n &  G_{y_{n}}& \tau_{x_1 y_{n}} & \cdots & \tau_{x_n y_{n}} 
\end{vmatrix}, \quad I=\begin{vmatrix}
     0 & 0 & w_{x_1} & \cdots & w_{x_n} \\
     0 &  0 & \tau_{x_1} & \cdots & \tau_{x_n} \\
     I_1 & G_{y_1}& \tau_{x_1 y_1} & \cdots & \tau_{x_n y_1} \\
    \vdots & \vdots & \vdots & \vdots & \vdots \\
    I_n &  G_{y_{n}}& \tau_{x_1 y_{n}} & \cdots & \tau_{x_n y_{n}} 
\end{vmatrix}.
\end{align*}

We first analyze the term $I$. We multiply each $i+2$-th column of $I$, $i=1,\cdots, n$, by $-\lambda^{-2} a^{ij}(x,\nabla_x\tau )w_{x_j}$ and add them to the first column, which does not change the determinant, thus
\begin{align*}
    I=\begin{vmatrix}
     -\lambda^{-2} a^{ij}(x,\nabla_x\tau)w_{x_j}w_{x_i} & 0 & w_{x_1} & \cdots & w_{x_n} \\
     -\lambda^{-2} a^{ij}(x,\nabla_x\tau)w_{x_j}\tau_{x_i} &  0 & \tau_{x_1} & \cdots & \tau_{x_n} \\
     w_{x_i}v^i_{y_1}-\lambda^{-2} a^{ij}(x,\nabla_x\tau)w_{x_j}\tau_{x_i y_1} & G_{y_1}& \tau_{x_1 y_1} & \cdots & \tau_{x_n y_1} \\
    \vdots & \vdots & \vdots & \vdots & \vdots \\
    w_{x_i}v^i_{y_n}-\lambda^{-2} a^{ij}(x,\nabla_x\tau)w_{x_j}\tau_{x_i y_n} &  G_{y_{n}}& \tau_{x_1 y_{n}} & \cdots & \tau_{x_n y_{n}} 
\end{vmatrix}.
\end{align*}
By \eqref{eq3}
$$ -\lambda^{-2} a^{ij}(x,\nabla_x\tau)w_{x_j}\tau_{x_i}=-w_{x_j}v^j=-\mathcal L w,$$
\begin{align*}
    w_{x_i}v^i_{y_k}-\lambda^{-2} a^{ij}(x,\nabla_x\tau)w_{x_j}\tau_{x_i y_k} & =\lambda^{-2} \big[\p_{y_k}a^{ij}(x,\nabla_x\tau)\big]w_{x_i}\tau_{x_j}\\
    & =\lambda^{-2} \big[\p_{\xi_\ell}a^{ij}(x,\nabla_x\tau)\big] \tau_{x_\ell y_k} w_{x_i}\tau_{x_j} =0.
\end{align*}
The last equality is due to the definition of the fundamental tensor $a^*(x,\xi)$. In particular,
$$\p_{\xi_\ell} a^{ij}(x,\nabla_x\tau)=\p_{\xi_j}a^{i\ell}(x,\nabla_x\tau).$$
Moreover, $a^*$ is homogeneous of degree $0$ in $\xi$, thus by the Euler's theorem
$$\p_{\xi_j}a^{i\ell}(x,\nabla_x\tau)\,\tau_{x_j}=0.$$

We denote
\begin{align*}
    B_{11}=\begin{vmatrix}
      0 & \tau_{x_1} & \cdots & \tau_{x_n} \\
     G_{y_1}& \tau_{x_1 y_1} & \cdots & \tau_{x_n y_1} \\
    \vdots & \vdots & \vdots & \vdots \\
   G_{y_{n}}& \tau_{x_1 y_{n}} & \cdots & \tau_{x_n y_{n}} 
    \end{vmatrix},\quad B_{21}=\begin{vmatrix}
          0 & w_{x_1} & \cdots & w_{x_n} \\
     G_{y_1}& \tau_{x_1 y_1} & \cdots & \tau_{x_n y_1} \\
    \vdots & \vdots & \vdots & \vdots \\
    G_{y_{n}}& \tau_{x_1 y_{n}} & \cdots & \tau_{x_n y_{n}} 
    \end{vmatrix}.
\end{align*}
The above calculation implies that
\begin{equation*}
    I = \mathcal{L}wB_{21} - B_{11} \lambda^{-2} a^{ij}(x,\nabla_x\tau)w_{x_i}w_{x_j}.
\end{equation*}

\begin{remark} 
In the Riemannian case, $a^{ij}(x,\xi)=a^{ij}(x)$, which simplifies the calculation, see \cite{Muhometov}. However, we are able to show that the argument for the Riemannian case can be carried over to the Finsler case, due to the homogeneity of the Finsler metric, which is crucial for our paper.
\end{remark}

Then again following \cite{Muhometov}, we have
\begin{align*}
    \int_{\p D}\sum_{i=1}^n D_i\, dS_y=\int_{\p D}(n-1)\mathcal L w B_{21}-(n-2)B\, dS_y.
\end{align*}
Based on the above arguments, \eqref{main2} becomes
\begin{align}\label{main3}
       2(-1)^n \int_{\p D}  B dS_y = (-1)^{n-1} \int_{\p D}  [ B_{11} \lambda^{-2}a^{ij}(x,\nabla_x\tau) w_{x_i}w_{x_j}+(n-2)\mathcal L w B_{21}-H+\sum_{i=1}^n C_i ] \,dS_y.
\end{align}

Now the proof will be completed by applying the following lemmas.

\begin{lemma}\label{auxilarylemma1}
Under the assumption of Lemma \ref{lemma2}, for fixed $x\in D$, 
    \begin{align*}
        (-1)^{n-1}\int_{\p D}  \big[(n-2)\mathcal L w B_{21}+ B_{11} \lambda^{-2}a^{ij}(x,\nabla_x \tau(x,y))w_{x_i}w_{x_j}\big] \,dS_y\\
        =\int_{S^*_x D} \eta(\frac{v}{|v|}) \big[(n-2) (\mathcal{L}w)^2+\lambda(x)^{-2} a^{ij}(x, \eta) w_{x_i}w_{x_j}\big]_{y= y(x,\eta)} dS_\eta 
    \end{align*}
\end{lemma}

Lemma \ref{auxilarylemma1} is essentially a change of variables.  We give a sketch of the proof below, see \cite{Muhometov} for a similar argument.

\begin{proof}[Sketch of the proof of Lemma \ref{auxilarylemma1}]
For a fixed $x\in D$, due to the simplicity assumption, given any $\eta\in S^*_x D$, there is a unique $y\in \p D$, such that $(x, v(y,x))=L^{-1}(x,\eta)$ (equivalently $\eta=\nabla_x\tau(x,y))$), which induces a diffeomorphism between $\p D$ and $S^*_x D$. We denote the inverse mapping by $y=y(x,\eta)$.
We use $\theta=(\theta_1,\cdots, \theta_{n-1})\in \mathbb S^{n-1}$ to parametrize both $\p D$ and $S^*_x D$ so that
$$\eta(\theta)=\nabla_x\tau(x,y(\theta)):=\nabla_x\tau (\theta).$$

We denote
\begin{align*}
    J_1 & = \int_S (n-2) \mathcal Lw B_{21}  dS_y,\\
    J_2& = \int_S B_{11} \lambda^{-2} a^{ij}(x,\nabla_x \tau(x,y))w_{x_i}w_{x_j} dS_y.
\end{align*}
We analyze $J_1$ first, after applying the parameter $\theta$
\begin{equation} \label{eq54}
J_1 = \int_{\mathbb S^{n-1}} (n-2) \mathcal Lw B_{21} |N| \bigg|_{y = y(\theta)} d\theta
\end{equation}
where $d\theta$ is the Euclidean measure on $\mathbb S^{n-1}$. The vector $N = [\frac{\partial y}{ \partial \theta_1},.....,\frac{\partial y}{ \partial \theta_{n-1}}]$, where $[...]$ denotes the vector product of vectors, so the i-th component of $N$ is given by
$$N_i =  (-1)^{i-1} \bigg| \frac{\p\, (y_1, \cdots, y_{i-1}, y_{i+1}, \cdots, y_n)}{\p\,(\theta_1,\cdots,\theta_{n-1})}\bigg|.$$
By properly choosing the orientation of the coordinates $\theta$, one can make $N$ outward normal to $\p D$ and verify that $|N|$ is the Jacobian of the transform $y(\theta)$, so
\begin{equation}\label{NsameasF}
   \nabla G(y(\theta)) = (-1)^n \frac{N}{|N|}.
\end{equation}
Recall that $G$ is the boundary defining function with $|\nabla G(y)|_{y\in\p D}=1$.
    

Now we claim that the following equality holds
\begin{equation}\label{B21}
    (-1)^{n-1}B_{21} |N| = \langle \frac{\tilde{N}}{|\tilde{N}|}, \nabla_x w \rangle |\tilde{N}|,
    \end{equation}
where $\tilde{N} = [\frac{\partial \nabla_x \tau }{\partial \theta_1},....,  \frac{\partial \nabla_x \tau }{\partial \theta_{n-1}} ] $. The equality \eqref{B21} will be proved as Lemma \ref{algebraiclemma}.
Similar to $N$, one can check that $\tilde{N}$ is outward normal to $S^*_xD$. On the other hand, $S^*_x D$ is a level set of the function $R(\xi)=\lambda^{-2}a^{ij}(x,\xi)\xi_i\xi_j$, therefore $\nabla_\xi R(\xi)$ is normal to $S^*_x D$ at $\xi\in S^*_x D$. Applying the definition of the fundamental tensor $a^*$ and the Euler's theorem, 
$$\p_{\xi_i} R(\xi)=\lambda^{-2}a^{ij}(x,\xi)\xi_j,$$
thus by \eqref{eq3}
$$\nabla_{\xi}R(\nabla_x\tau(x,y))=v(y,x).$$
It follows that
\begin{equation}\label{Ntov}
    \frac{\tilde N}{|\tilde N|}=\frac{v}{|v|}=\eta (\frac{v}{|v|}) \,v,
\end{equation}
where the last equality is due to \eqref{eq2}.
Combining \eqref{eq54}, \eqref{B21} and \eqref{Ntov}, we obtain
$$J_1 = (-1)^{n-1}\int_{\mathbb S^{n-1}} (n-2) \eta(\frac{v}{|v|}) (\mathcal Lw)^2 |\tilde{N}| \bigg|_{y=y(\theta)}\, d\theta = (-1)^{n-1}\int_{S^*_x D} (n-2) \eta (\frac{v}{|v|}) (\mathcal Lw)^2 \bigg|_{y=y(x,\eta)}\, d\eta,$$
where $\eta = \nabla_x\tau(\theta)$. Here we used the fact that $|\tilde{N}|$ is the Jacobian of the change of coordinates $\eta =\eta (\theta)$.

Similar to \eqref{B21}, one can check that 
$$(-1)^{n-1} B_{11} |N|=\langle \tilde N, \nabla_x \tau\rangle,$$
see the proof of Lemma \ref{algebraiclemma} for more details.
Thus we may transform the integral $J_2$ analogously to get
$$J_2 = (-1)^{n-1}\int_{S^*_x D} \eta(\frac{v}{|v|}) \lambda(x)^{-2} a^{ij}(x,\eta) w_{x_i}w_{x_j} \bigg|_{y=y(x,\eta)}  \,dS_\eta.$$

This completes the proof.
\end{proof}

Now we justify the equality \eqref{B21}, whose proof is purely algebraic.

\begin{lemma}\label{algebraiclemma}
Under the same notations from the proof of Lemma \ref{auxilarylemma1}
    $$(-1)^{n-1}B_{21} |N|=\langle \tilde{N}, \nabla_x w \rangle.$$
\end{lemma}
\begin{proof} 
By definition 
    \begin{align*}
       B_{21} |N| &=(-1)^n\begin{vmatrix}
           0 & w_{x_1} & \cdots & w_{x_n} \\
           N_1 & \tau_{x_1 y_1} & \cdots & \tau_{x_n y_1} \\
           \vdots & \vdots & \ddots & \vdots \\
           N_n & \tau_{x_1 y_n} & \cdots & \tau_{x_n y_n} 
       \end{vmatrix} \\
       & =(-1)^n\sum_{i=1}^n w_{x_i} (-1)^i \sum_{j=1}^n (-1)^{j+1} N_j \cdot \det Q_{ji}
    \end{align*}
where $Q_{ji}$ is the matrix obtained by removing the $j$-th row and the $i$-th column of the matrix $\begin{pmatrix}
    \tau_{x_1 y_1} & \cdots & \tau_{x_n y_1} \\
           \vdots & \ddots & \vdots \\
          \tau_{x_1 y_n} & \cdots & \tau_{x_n y_n}
\end{pmatrix}$. We will show that $\sum_{j=1}^n (-1)^{i+j}N_j \cdot \det Q_{ji}=\tilde N_i$.

Recall that
\begin{align*}
    \tilde N_i & =(-1)^{i+1}\begin{vmatrix}
        \frac{\partial \tau_{x_1}}{\partial \theta_1} & \cdots & \frac{\partial \tau_{x_{i-1}}}{\partial \theta_1} & \frac{\partial \tau_{x_{i+1}}}{\partial \theta_1} & \cdots & \frac{\partial \tau_{x_n}}{\partial \theta_1} \\
        \vdots & \vdots & \vdots & \vdots & \vdots & \vdots\\
        \frac{\partial \tau_{x_1}}{\partial \theta_{n-1}} & \cdots & \frac{\partial \tau_{x_{i-1}}}{\partial \theta_{n-1}} & \frac{\partial \tau_{x_{i+1}}}{\partial \theta_{n-1}} & \cdots & \frac{\partial \tau_{x_n}}{\partial \theta_{n-1}} 
    \end{vmatrix}
\end{align*}
    Note that by the chain rule
    \begin{align*}
        \begin{pmatrix}
        \frac{\partial \tau_{x_1}}{\partial \theta_1} & \cdots & \frac{\partial \tau_{x_{i-1}}}{\partial \theta_1} & \frac{\partial \tau_{x_{i+1}}}{\partial \theta_1} & \cdots & \frac{\partial \tau_{x_n}}{\partial \theta_1} \\
        \vdots & \vdots & \vdots & \vdots & \vdots & \vdots\\
        \frac{\partial \tau_{x_1}}{\partial \theta_{n-1}} & \cdots & \frac{\partial \tau_{x_{i-1}}}{\partial \theta_{n-1}} & \frac{\partial \tau_{x_{i+1}}}{\partial \theta_{n-1}} & \cdots & \frac{\partial \tau_{x_n}}{\partial \theta_{n-1}} 
    \end{pmatrix}=P\cdot Q_i
    \end{align*}
    where 
    $$P=\begin{pmatrix}
        \frac{\partial y_1}{\partial \theta_1} & \cdots & \frac{\partial y_n}{\partial \theta_1}\\
        \vdots & \ddots & \vdots\\
        \frac{\partial y_1}{\partial \theta_{n-1}} & \cdots & \frac{\partial y_n}{\partial \theta_{n-1}}\\
    \end{pmatrix}_{(n-1)\times n}$$
    and 
    $$Q_i=\begin{pmatrix}
        \frac{\partial \tau_{x_1}}{\partial y_1} & \cdots & \frac{\partial \tau_{x_{i-1}}}{\partial y_1} & \frac{\partial \tau_{x_{i+1}}}{\partial y_1} & \cdots & \frac{\partial \tau_{x_n}}{\partial y_1}\\
        \vdots & \vdots & \vdots & \vdots & \vdots & \vdots\\
        \frac{\partial \tau_{x_1}}{\partial y_n} & \cdots & \frac{\partial \tau_{x_{i-1}}}{\partial y_n} & \frac{\partial \tau_{x_{i+1}}}{\partial y_n} & \cdots & \frac{\partial \tau_{x_n}}{\partial y_n}
    \end{pmatrix}_{n\times (n-1)}$$
  Now by the Cauchy-Binet formula
    \begin{align*}
        \det (P \cdot Q_i)=\sum_{j=1}^n \det P_j\cdot \det Q_{ji}
    \end{align*}
    where $P_j$ is the square matrix obtained by removing the $j$-th column of $P$ and $Q_{ji}$ is the square matrix obtained as before. Recall that $N_j=(-1)^{j+1}\det P_j$, therefore
    $$\tilde N_i=(-1)^{i+1}\sum_{j=1}^n (-1)^{j+1} N_j\cdot \det Q_{ji}=\sum_{j=1}^n (-1)^{i+j} N_j \cdot \det Q_{ji}.$$
    It follows that
    $$B_{21} |N|=(-1)^n\sum_{i=1}^n w_{x_i} (-\tilde N_i)=(-1)^{n-1} \langle \tilde{N}, \nabla_x w \rangle.$$
\end{proof}

\begin{lemma}\label{auxilarylemma2}
For fixed $y\in \p D$,
    $$\int_D [\sum_{i}^n C_i-H]\, dx=-\int_{\p D} B(w,w,G,\tau)\,dS_x.$$
\end{lemma}

\begin{proof}
    By purely algebraic calculations, fixing $y\in \p D$ (so $G(y)=0$), it was shown in \cite{Muhometov} that
$$\int_D (\sum_{i=1}^n C_i-H)\, dx=\int_D \sum_{i,j,k=1}^n (-1)^{i+j+k}\big[w_{x_j} w_{y_k}P_{kij}\sigma_{ij}\big]_{x_i}\,dx,$$
where $P_{kij}$ is the minor obtained from $B_{11}$ by removing the first and ($k+1$)-th row, and the ($i+1$)-th and ($j+1$)-th column, and
$$\sigma_{ij} =  \begin{cases}
     1, \ \ i < j, \\
    -1, \ \ i>j.
\end{cases}$$
Here one needs to use \eqref{ellipsoid}, and differentiate it w.r.t. $y_k$ to get
$$\lambda^{-2}(x) a^{ij}(x,\nabla_x\tau(x,y)) \tau_{x_i y_k}\tau_{x_j}=0.$$
Recall that due to the homogeneity of the fundamental tensor and the Euler's theorem, terms involving derivatives of $a^{ij}$ vanish.

Now applying the divergence theorem, we obtain
\begin{equation}\label{integralbydivergence}
\begin{split}
    \int_D (\sum_{i=1}^n C_i-H)\, dx & = \int_D \sum_{i,j,k=1}^n (-1)^{i+j+k}\big[w_{x_j} w_{y_k}P_{kij}\sigma_{ij}\big]_{x_i}\, dx\\
    & = \int_{\p D} \sum_{i,j,k=1}^n (-1)^{i+j+k} w_{x_j}w_{y_k} P_{kij}\sigma_{ij}G_{x_i} \,dS_x.  
\end{split}
\end{equation}
Since the integrand in the integral over $D$ has a singularity at $x = y$ on the boundary $\p D$, we need to justify the use of divergence theorem above. One may take a small $\epsilon$-ball $U_\epsilon$ around $y$ so that $D \setminus U_\epsilon$ has piecewise smooth boundary, apply the divergence theorem to the integral over $D \setminus U_\epsilon$, and let $U_\epsilon$ shrink to $y$, so \eqref{integralbydivergence} is justified.

By expanding along the first and second row, then the first column of $B(w,w,G,\tau)$, one can check that 
$$\sum_{i,j,k=1}^n (-1)^{i+j+k} w_{x_j}w_{y_k} P_{kij}\sigma_{ij}G_{x_i}=-B(w,w,G,\tau).$$
This finishes the proof.
\end{proof}


Finally, combining \eqref{main3}, Lemma \ref{auxilarylemma1} and Lemma \ref{auxilarylemma2}, the proof of Lemma \ref{lemma2} is complete.\\
\\

\noindent {\bf AI disclosure:} Artificial intelligence was not used for this article.

\bibliographystyle{plain}
\bibliography{references}

\end{document}